\documentclass{amsart}

\newtheorem{thm}{Theorem}[section]
\newtheorem{cor}[thm]{Corollary}

\newtheorem{ques}[thm]{Question}
\newtheorem{ex}[thm]{Example}

\newtheorem{prop}[thm]{Proposition}
\theoremstyle{definition}

\theoremstyle{remark}
\newtheorem{rem}[thm]{Remark}
\numberwithin{equation}{section}

\begin{document}

\title[Sums of squares over totally real fields]
{Sums of squares over totally real \\ fields are rational sums of squares}%
\author{Christopher J. Hillar}%
\address{Department of Mathematics, Texas A\&M University, College Station, TX 77843}
\email{chillar@math.tamu.edu}
 \thanks{Supported under a National Science Foundation Postdoctoral Research Fellowship.}
  

\begin{abstract}
Let $K$ be a totally real number field with Galois closure $L$.
We prove that if $f \in \mathbb Q[x_1,\ldots,x_n]$ is a sum of 
$m$ squares in $K[x_1,\ldots,x_n]$, then $f$ is a sum of 
\[4m \cdot 2^{[L: \mathbb Q]+1} {[L: \mathbb Q] +1 \choose 2}\] 
squares in $\mathbb Q[x_1,\ldots,x_n]$.  Moreover, our argument
is constructive and generalizes to the case of commutative $K$-algebras.
This result gives a partial resolution to a question of Sturmfels
on the algebraic degree of certain semidefinite programing problems.  
\end{abstract}
\maketitle

\section{Introduction}

In recent years, techniques from semidefinite programming have produced numerical algorithms 
for finding representations of positive semidefinite polynomials 
as sums of squares.  These algorithms have many applications in optimization, control theory,
quadratic programming, and matrix analysis \cite{PS,Parrilo, Parrilo2, SOSTOOLS, SOSTOOLS2}.  For a
noncommutative application of these techniques to a famous trace
conjecture, see the papers \cite{Burgdorf,Haegele, Schweighofer,LS} which continue on the work
of \cite{Hi}. 

One major drawback with these algorithms is that their output is, in general, numerical.
For many applications, however, exact polynomial identities are needed. 
In this regard, Sturmfels has asked whether a representation with real
coefficients implies one over the rationals.

\begin{ques}[Sturmfels]\label{sturmfelsques}
If $f \in \mathbb Q[x_1,\ldots,x_n]$ is a sum of squares in $\mathbb R[x_1,\ldots,x_n]$,
then is $f$ also a sum of squares in $\mathbb Q[x_1,\ldots,x_n]$?
\end{ques}

It is well-known that a polynomial is a sum of real polynomial squares if and only if it 
can be written in the form 
\begin{equation}\label{fBeq}
f = \textbf{v}^TB\textbf{v},
\end{equation}
in which $\textbf{v}$ is a column vector of monomials and $B$ is a real positive semidefinite
(square) matrix \cite{powers}; in this case, the matrix $B$ is called a \textit{Gram matrix} for $f$.  If $B$ happens to have rational entries, then $f$ is a sum of squares in $\mathbb Q[x_1,\ldots,x_n]$ (this follows from a  Cholesky factorization argument or from a matrix generalization of Lagrange's four square theorem \cite{HN}).  Thus, in the language of quadratic forms, Sturmfels is asking whether the existence of  a positive semidefinite Gram matrix for $f \in \mathbb Q[x_1,\ldots,x_n]$ over the reals implies that one exists over the rationals. 

Although a complete answer to Question \ref{sturmfelsques}  is not known, Parrilo and Peyrl have
written an implementation of SOSTOOLS in the algebra package Macaulay 2 that
attempts to find rational representations of polynomials that are sums of squares \cite{ParriloPeyrl}.
Their idea is to approximate a real Gram matrix $B$ with rational numbers and then project back to 
the linear space of solutions governed by equation (\ref{fBeq}).

The following result says that Question \ref{sturmfelsques} has a positive answer in
some ``generic" sense;  it also explains the difficulty of finding counterexamples.

\begin{thm}
Let $f \in \mathbb Q[x_1,\ldots,x_n]$.  If there is an invertible Gram matrix $B$ for $f$, 
then there is a Gram matrix for $f$ with rational entries.
\end{thm}
\begin{proof}
Let $B$ be a real positive semidefinite matrix and $\textbf{v}$ a vector of monomials
such that  $f = \textbf{v}^TB\textbf{v}$.  Consider the set $L$ of real symmetric matrices 
$S = S^T = (s_{ij})$ of the same size as $B$ for which $f = \textbf{v}^TS\textbf{v}$.
This space corresponds to the solutions of a set of linear equations in the $s_{ij}$ over $\mathbb Q$.
From elementary linear algebra (Gaussian elimination), it follows that there is an integer
$k$  such that 
\[ L = \{S_0 +  t_1 S_1+ \cdots +t_kS_k: t_1,\ldots,t_k \in \mathbb R\}\]
for some rational symmetric matrices $S_1,\ldots,S_k$.
The subset of matrices in $L$ that are positive definite is determined by a finite set of strict 
polynomial inequalities in the $t_1,\ldots,t_k$ produced by
setting all the leading principal minors to be positive \cite[p. 404]{HJ1}.  
By continuity, a real positive definite solution $B$ guarantees a rational one,
and this completes the proof.
\end{proof}

\begin{rem}
The argument above shows that we may find a rational Gram matrix of the same size as the 
original Gram matrix $B$.  Is this true even if $B$ is not invertible?  We suspect not.
\end{rem}

Although the general case seems difficult, 
Question \ref{sturmfelsques} has a positive answer for univariate polynomials due to results of Landau 
\cite{landau}, Pourchet \cite{Pourchet}, and (algorithmically) Schweighofer \cite{Schweighofer2}.  
In fact, Pourchet has shown that at most $5$ polynomial squares
in $\mathbb Q[x]$ are needed to represent every positive semidefinite polynomial in $\mathbb Q[x]$, 
and this is best possible.

It follows from Artin's solution to Hilbert's $17$th problem \cite[Theorem 2.1.12]{PosPoly}
that if \mbox{$f \in \mathbb Q[x_1,\ldots,x_n]$} is a sum of 
squares of rational functions in $\mathbb R(x_1,\ldots,x_n)$, then it is a sum
of squares in $\mathbb Q(x_1,\ldots,x_n)$.  Moreover, from the work of 
Voevodsky on the Milnor conjectures, it is known that $2^{n+2}$ such squares suffice \cite[p. 530]{LAM}.  
However, the transition from rational functions to polynomials is often a very delicate one.  For instance,
not every polynomial that is a sum of squares of rational functions 
is a sum of squares of polynomials \cite[p. 398]{LAM}. 

More generally, Sturmfels is interested in the algebraic degree \cite{NRS} of 
maximizing a linear functional over the space of all sum of 
squares representations of a given polynomial that is a sum of squares.
In the special case of Question \ref{sturmfelsques}, a positive answer 
signifies an algebraic degree of $1$ for this optimization problem.

General theory (for instance, Tarski's Transfer Principle for real closed fields 
\cite[Theorem 2.1.10]{PosPoly}) reduces Question \ref{sturmfelsques} to one involving real algebraic numbers.  In this paper, we present a positive answer to 
this question for a special class of fields.

Recall that a \textit{totally real number field} is a finite algebraic extension of
$\mathbb Q$ all of whose complex embeddings lie entirely in $\mathbb R$.
For instance, the field $\mathbb Q(\sqrt{d})$ is totally real for positive, integral $d$.
Our main theorem is the following.

\begin{thm}\label{mainthm}
Let $K$ be a totally real number field with Galois closure $L$.
 If $f \in \mathbb Q[x_1,\ldots,x_n]$ is a sum of 
$m$ squares in $K[x_1,\ldots,x_n]$, then $f$ is a sum of 
\[4m \cdot 2^{[L: \mathbb Q]+1} {[L: \mathbb Q] +1 \choose 2}\] 
squares in $\mathbb Q[x_1,\ldots,x_n]$.
\end{thm}

Our techniques also generalize naturally to the following situation.
Let $R$ be a commutative $\mathbb Q$-algebra and let $K$ be a totally real 
number field.  Also, set $S := R \otimes_{\mathbb Q} K$, which we naturally
identify as a ring extension of $R = R \otimes_{\mathbb Q} \mathbb Q$.
If $f$ is a sum of the form
\begin{equation*}\label{ratsos}
f = \sum_{i=1}^m p_i^2, \ \ \ p_i \in  R,
\end{equation*}
then we say that $f$ is a \textit{sum of squares} over $R$.  
It is a difficult problem to determine those $f$ which are sums of squares over $R$.  
In this setting, Theorem \ref{mainthm} generalizes in the following way.

\begin{thm}\label{mainthm2}
Let $K$ be a totally real number field with Galois closure $L$.
If $f \in R$ is a sum of 
$m$ squares in $R \otimes_{\mathbb Q} K$, then $f$ is a sum of 
$4m \cdot 2^{[L: \mathbb Q]+1} {[L: \mathbb Q] +1 \choose 2}$ squares over $R$.
\end{thm}

\begin{rem}
One can view Theorem \ref{mainthm2}
as a ``going-down" theorem \cite{EL} for certain quadratic forms over the rings $R$ and 
$R \otimes_{\mathbb Q} K$.  We do not know how much the factor
$2^{[L: \mathbb Q]+1} {[L: \mathbb Q] +1 \choose 2}$ can be improved upon, although
we suspect that for polynomial rings, it can be improved substantially.
We remark that it is known \cite{CDLR} that 
arbitrarily large numbers of squares are necessary to represent
any sum of squares over $\mathbb R[x_1,\ldots,x_n]$, $n > 1$, making a fixed
bound (for a given $n$) as in the rational function case impossible.
\end{rem}

Our proof of Theorem \ref{mainthm} is also constructive.


\begin{ex}\label{cubeex}
Consider the polynomial \[f = 3- 12y - 6x^3+18y^2+3x^6+12x^3y-6xy^3+6x^2y^4.\]
This polynomial is a sum of squares over $\mathbb R[x,y]$.  To see this,
let $\alpha,\beta,\gamma \in \mathbb R$ be the roots of the polynomial 
$u(x) = x^3-3x+1$.  Then, a computation reveals that
\[f = (x^3+\alpha^2 y + \beta xy^2-1)^2+(x^3+\beta^2 y + \gamma xy^2-1)^2+(x^3+\gamma^2 y + \alpha xy^2-1)^2.\]
Using our techniques, we can construct from this representation one over $\mathbb Q$:
\begin{equation*}
\begin{split}
\left( x^3+xy^2+3y/2 - 1 \right) ^{2}+ & \left( x^3+2y - 1 \right) ^{2}+ \left( x^3-xy^2+5y/2 - 1 \right) ^{2} \\
& + \left( 2y - xy^2 \right) ^{2}+3 y^2/2+3x^2y^4.\\
\end{split}
\end{equation*}
This example will be revisited many times in the sequel to illustrate our proof.  
\qed
\end{ex}

We believe that in Theorem \ref{mainthm} the field 
$K$ may be replaced by any real algebraic extension of the rationals
(thus giving a positive answer to Sturmfels' question); 
however, our techniques do not readily generalize to this situation.
We shall discuss the obstructions throughout our presentation.


The organization of this paper is as follows.
In Section \ref{prelim}, we set up our notation and state a weaker (but still sufficient) version of 
our main theorem.  Section \ref{vandfact} describes a matrix factorization
for Vandermonde matrices.  This construction is applied in the subsequent section to reduce the
problem to the case $K = \mathbb Q(\sqrt{l_1},\ldots,\sqrt{l_r})$ for positive
integers $l_k$.  Finally, the proof of Theorem \ref{mainthm} is completed in Section \ref{proof}.
For simplicity of exposition, we shall focus on the polynomial version of our main theorem
although it is an easy matter to translate the techniques to prove the more general 
Theorem \ref{mainthm2}.

We would like to thank T. Y. Lam, Bruce Reznick, and Bernd Sturmfels for interesting
discussions about this problem.  We also thank the anonymous referee for 
several suggestions that improved the exposition of this work.

\section{Preliminaries}\label{prelim}

An equivalent definition of a totally real number field $K$ is that it is a field generated by 
a root of an irreducible polynomial $u(x) \in \mathbb Q[x]$, all of whose zeroes are real.
For instance, the field $K = \mathbb Q(\alpha,\beta,\gamma) = \mathbb Q(\alpha)$ arising in 
Example \ref{cubeex} is a totally real (Galois) extension of $\mathbb Q$ in this sense.
A splitting field of $u(x)$ (a Galois closure of $K$) is also totally real, so we 
lose no generality in assuming  that $K$ is a totally real Galois extension of $\mathbb Q$.
We will therefore assume from now on that  $K = \mathbb Q(\theta)$ is Galois and that
$\theta$ is a real algebraic number, all of whose conjugates are also real.  
We set $r = [K: \mathbb Q]$ and let $G$ be the Galois group Gal$(K/\mathbb Q)$.
For the rest of our discussion, we will fix $K =  \mathbb Q(\theta)$ with these parameters.

We begin by stating a weaker formulation of Theorem \ref{mainthm}.  For the purposes of 
this work, a \textit{rational sum of squares} is a linear combination of squares
with positive rational coefficients.

\begin{thm}\label{weakmainthm}
Let $K$ be a totally real number field that is Galois over $\mathbb Q$.
Then for any $p \in K[x_1,\ldots,x_n]$, the polynomial
\[ f = \sum_{\sigma \in G} (\sigma p)^2\]
can be written as a rational sum of 
$2^{[K: \mathbb Q]+1} {[K: \mathbb Q] +1 \choose 2}$ squares in $\mathbb Q [x_1,\ldots,x_n]$.
\end{thm}

\begin{rem} 
Elements of the form $\sum_{\sigma \in G} (\sigma p)^2$
are also sometimes called \textit{trace forms} for the field extension $K$ \cite[p. 217]{LAM}.

\end{rem}


It is elementary, but important that this result implies Theorem \ref{mainthm}.

\begin{proof}[Proof of Theorem \ref{mainthm}]
Let $f = \sum_{i=1}^m p_i^2 \in \mathbb Q [x_1,\ldots,x_n]$ be 
a sum of squares with each $p_i \in K[x_1,\ldots,x_n]$. 
Summing both sides of this equation over all actions of \mbox{$G = \text{Gal}(K/\mathbb Q)$}, we have
\[ f = \frac{1}{|G|} \sum_{i=1}^m \sum_{\sigma \in G} (\sigma p_i)^2.\]
The conclusions of Theorem \ref{mainthm}
now follow immediately from Theorem \ref{weakmainthm} and Lagrange's four
square theorem (every positive rational number is the sum of at most four squares).
\end{proof}

\begin{rem}
This averaging argument can also be found in the papers \cite{CDLR2,GP}.
\end{rem}


We will focus our remaining efforts, therefore, on proving Theorem \ref{weakmainthm}.

\section{Vandermonde Factorizations}\label{vandfact}

To prepare for the proof of Theorem \ref{weakmainthm}, we 
describe a useful matrix factorization. 
It is inspired by Ilyusheckin's recent proof \cite{NVI} that the 
discriminant of a symmetric matrix of indeterminates is a sum of squares, although
it would not surprise us if the factorization was known much earlier.

Let $A = A^T$ be an $r \times r$ symmetric matrix over a field $F$ of characteristic 
not equal to $2$, and let
$y_1,\ldots,y_r$ be the eigenvalues of $A$ in an algebraic closure of $F$.
Also, let $V_r$ be the Vandermonde matrix
\[ V_r = \left[\begin{array}{cccc}1 & 1 & \cdots & 1 \\y_1 & y_2 & \cdots & y_r \\\vdots & \vdots & \ddots & \vdots \\y_1^{r-1} & y_2^{r-1} & \cdots & y_r^{r-1}\end{array}\right].\]
The matrix $B = V_rV_r^T$ has as its $(i,j)$th entry the $(i+j-2)$th Newton power sum of
the eigenvalues of $A$: \[ \sum_{k=1}^r y_k^{i+j-2}.\] 
Since the trace of $A^m$ is also the $m$th Newton power sum of the $y_k$, it follows that
we may write $B = [\text{tr}(A^{i+j-2})]_{i,j=1}^{r} \in F^{r \times r}$.  We next give another factorization 
of $B$ in the form $CC^T$.

Let $E_{ij}$ be the $r \times r$ matrix with a $1$ in the $(i,j)$ entry and 
$0$'s elsewhere.  A basis for $r \times r$ symmetric matrices is then given by the 
following ${r+1 \choose 2}$ matrices:
\[\{E_{ii}: i = 1,\ldots,r\} \ \cup \  \{(E_{ij} + E_{ji})/\sqrt{2}: 1 \leq i < j \leq r\}.\]
For example, the ``generic" symmetric $2 \times 2$ matrix
\begin{equation}\label{2by2A}
A = \left[\begin{array}{cc}x_{11} & x_{12} \\x_{12} & x_{22} \end{array}\right],
\end{equation}
with entries in the field $F = \mathbb Q(x_{11},x_{12},x_{22})$,
is represented in this basis as
\[x_{11}\left[\begin{array}{cc}1 & 0 \\0 & 0\end{array}\right]
+ x_{22} \left[\begin{array}{cc} 0 & 0 \\0 & 1\end{array}\right]
+ \sqrt{2} \cdot  x_{12} \left[\begin{array}{cc}0 & 1/\sqrt{2} \\1/\sqrt{2} & 0 \end{array}\right].\]
This basis is useful since the inner product of two symmetric matrices $P$ and $Q$
with respect to this basis is simply tr$(PQ)$, as one can easily check.

Express the powers $A^m$ in terms of this basis and place 
the vectors of the coefficients as rows of a matrix $C$.  The entries of the $r \times {r+1 \choose 2}$
matrix $C$ will be in $F[\sqrt{2}]$.
Our construction proves the formal identity
\begin{equation}\label{matrixfactor} 
V_rV_r^T = [\text{tr}(A^{i+j-2})]_{i,j=1}^r = CC^T.
\end{equation}

\begin{ex}  
With $A$ given by (\ref{2by2A}), the factorization reads:
\begin{equation*}\label{C}
 \left[\begin{array}{cc}1 & 1 \\y_1 & y_2\end{array}\right]\left[\begin{array}{cc}1 & y_1 \\ 1 & y_2\end{array}\right]
 = \left[\begin{array}{ccc}1 & 1 & 0 \\x_{11} & x_{22} & \sqrt{2}\cdot x_{12}\end{array}\right]\left[\begin{array}{cc}1 & x_{11} \\1 & x_{22} \\0 & \sqrt{2}\cdot x_{12}\end{array}\right].
\end{equation*}
Algebraically, this equation reflects the fact that for a $2 \times 2$ symmetric matrix $A$, 
\begin{equation*}
\begin{split}
\text{\rm tr}(A) = \ & x_{11}+ x_{22},\\
\text{\rm tr}(A^2) = \ & \text{\rm tr}(A)^2 - 2\det(A) = x_{11}^2+x_{22}^2 + 2x_{12}^2. \\
\end{split}
\end{equation*}\qed
\end{ex}

In the next section, we will use the matrix factorization (\ref{matrixfactor}) to replace a Gram
matrix over $\mathbb Q(y_1,\ldots,y_r)$ with one over a much smaller field.  

\section{Symmetric matrices with prescribed characteristic polynomial}

Let $K = \mathbb Q(\theta)$ be totally real and Galois, and set 
$\sigma_1,\ldots,\sigma_r$ to be the elements of Gal$(K/\mathbb Q)$.
Given $p \in K[x_1,\ldots,x_n]$, we may express it in the form \[p = \sum_{i=0}^{r-1} q_{i} \theta^i,\]
for elements $q_i \in \mathbb Q [x_1,\ldots,x_n]$.  With this parameterization, the sum 
\[ \sum_{\sigma \in G} (\sigma p)^2 = 
\sum_{j=1}^{r} \left(\sum_{i=0}^{r-1} q_{i} (\sigma_j \theta)^i \right)^2\]
appearing in the statement of Theorem \ref{weakmainthm} may be written
succinctly as
\begin{equation}\label{vandtheta}  
\left[\begin{array}{c}q_0 \\\vdots \\q_{r-1}\end{array}\right]^T  \left[\begin{array}{cccc}1 &  \cdots & 1 \\ \sigma_1 \theta &  \cdots & \sigma_r \theta \\\vdots & \ddots  & \vdots \\ (\sigma_1\theta)^{r-1} &  \cdots & (\sigma_{r} \theta)^{r-1}\end{array}\right]
\left[\begin{array}{cccc}1 & \sigma_1 \theta & \hdots & (\sigma_1 \theta)^{r-1} \\\vdots & \vdots & \ddots & \vdots  \\1 & \sigma_r \theta & \hdots & (\sigma_r \theta)^{r-1}\end{array}\right]
\left[\begin{array}{c}q_0 \\\vdots \\q_{r-1}\end{array}\right].
\end{equation}

Let $V_r$ be the Vandermonde matrix appearing in equation (\ref{vandtheta}).
We would like to construct a factorization as in (\ref{matrixfactor}) to replace the 
elements of $K$ with numbers from $\mathbb Q(\sqrt{2})$.
To apply the techniques of Section \ref{vandfact}, however, we must 
find an $r \times r$ \textit{symmetric} matrix $A$
whose eigenvalues are $\sigma_1 \theta,\ldots,\sigma_r \theta$
(the roots of the minimal polynomial for $\theta$ over $\mathbb Q$).  A necessary 
condition is that these numbers are all real, but we would like a converse.
Unfortunately, a converse with matrices over $\mathbb Q$ is impossible.  
For the interested reader, we include a proof of this basic fact.

\begin{prop}
There is no rational, symmetric matrix with characteristic polynomial $u(x) = x^2-3$.
\end{prop}
\begin{proof}
We argue by way of contradiction. Let  \[A = \left[\begin{array}{cc}a & b \\ b & c\end{array}\right],  \ \ a,b,c \in \mathbb Q,\] and suppose that
\[u(x) = \det(xI-A) = x^2-(a+c)x + (ac-b^2).\]   It follows that there are rational numbers
$a$ and $b$ such that $a^2 + b^2 = 3$.  Multiplying by a common denominator,
one finds that there must be integer solutions $u,v,w$ to the diophantine equation
\begin{equation}\label{diopheq} 
u^2 + v^2 = 3w^2.
\end{equation}
Recall from elementary number theory that a number $n$ is the 
sum of two integral squares if and only if every prime $p \equiv 3  \ (\text{\rm mod } 4)$ that appears in 
the prime factorization of $n$ appears  to an even power.  This contradiction
finishes the proof.
\end{proof}


If we allow $A$ to contain square roots of rational numbers, however, then
there is always such a symmetric $A$.  This is the content of a result of
Fiedler \cite{Fiedler}.  We include his proof for completeness.

\begin{thm}[Fiedler]\label{symcomp}
Let $u(x) \in \mathbb C[x]$ be monic of degree $r$ and let $b_1,\ldots,b_r$ be distinct 
complex numbers such that $u(b_k) \neq 0$ for each $k$.  Set $v(x) = \prod_{k=1}^r(x-b_k)$ and
choose any complex numbers $d_1,\ldots,d_r$ and $\delta$ that satisfy \[ \delta v'(b_k)d_k^2 - u(b_k) = 0, \ \ k = 1,\ldots,r.\]  Let $d = [d_1,\ldots,d_r]^T$ and $B = \text{\rm diag}(b_1,\ldots,b_r)$.  Then the symmetric 
matrix \[ A = B - \delta dd^T\] has characteristic polynomial equal to $u(x)$. 
\end{thm}

\begin{proof}
Applying the Sherman-Morrison formula \cite[p. 50]{SM} for the determinant of a rank $1$ perturbation 
of a matrix, we have 
\begin{equation}
\begin{split}
\det(xI-A) = \ & \det(xI-B) + \delta \det(xI-B) d^T (xI-B)^{-1}d \\
= \ & \prod_{k=1}^r (x-b_k) + \delta \sum_{k=1}^r {d_k^2 \prod_{i=1, i \neq k}^r (x-b_i)}.
\end{split}
\end{equation}
Since the monic polynomial $\det(xI-A)$ and $u(x)$ agree for $x = b_1,\ldots,b_r$, it
follows that they are equal.
\end{proof}

\begin{rem}
There are simpler, tridiagonal matrices which can replace
the matrix $A$ (see \cite{Schmeisser}); however, square roots are still necessary to
construct them.
\end{rem}

The following corollary allows us to form a real symmetric matrix with
characteristic polynomial equal to the minimal polynomial for $\theta$
over $\mathbb Q$.

\begin{cor}\label{mainrealcor}
If $u(x) \in \mathbb Q[x]$ is monic of degree $r$ and has $r$ distinct real roots,
then there are positive rational numbers $l_1,\ldots,l_r$ and a symmetric matrix $A$ with entries in 
$\mathbb Q(\sqrt{l_1},\ldots,\sqrt{l_r})$ such that the eigenvalues of $A$ are the roots of $u(x)$.
\end{cor}

\begin{proof}
Let $b_1,\ldots,b_{r-1}$ be rational numbers such that exactly one $b_i$ is (strictly) between
consecutive roots of $u(x)$, and let $b_r$ be a rational number either 
smaller than the least root of $u(x)$ or
greater than the largest root of $u(x)$.  Also, set $\delta \in \{-1,1\}$
such that $l_k = \delta u(b_k)/v'(b_k)$ is positive for each $k$.  The corollary now
follows from Theorem \ref{symcomp} by setting $d_k = \sqrt{l_k}$ for each $k$.
\end{proof}

\begin{ex}\label{cubeexA}
Consider the polynomial $u(x) = x^3-3x+1$ from Example \ref{cubeex}.  Choosing
$(b_1,b_2,b_3) = (0,1,2)$ and $\delta = 1$, we have $d = [\sqrt{2}/2,1,\sqrt{6}/2]^T$ and
\[ A = \left[\begin{array}{ccc}-1/2 & -\sqrt{2}/2 & -\sqrt{3}/2 \\-\sqrt{2}/2 & 0 & -\sqrt{6}/2 \\-\sqrt{3}/2 & -\sqrt{6}/2 & 1/2\end{array}\right].\]
One can easily verify that the characteristic polynomial of $A$ is $u(x)$.\qed
\end{ex}

Combining Corollary \ref{mainrealcor} and the construction found in Section \ref{vandfact},
we have proved the following theorem.

\begin{thm}\label{almosttherethm}
Let $K$ be a totally real Galois extension of $\mathbb Q$ and set $r = [K: \mathbb Q]$.
Then for any $p \in K [x_1,\ldots,x_n]$, there are positive integers $l_1,\ldots,l_r$
such that 
\[ \sum_{\sigma \in G} (\sigma p)^2 = q^TCC^Tq,\]
in which $q$ is a vector of polynomials in $\mathbb Q [x_1,\ldots,x_n]$ and
$C$ is an $r \times {r+1 \choose 2}$ matrix with entries in 
$F = \mathbb Q(\sqrt{l_1},\ldots,\sqrt{l_r},\sqrt{2})$.
\end{thm}

To illustrate the computations performed in the proof of Theorem \ref{almosttherethm},
we present the following.

\begin{ex}
We continue with Example \ref{cubeexA}.  Let $\alpha \in \mathbb R$ be the 
root of $u(x)$ with $\alpha \in (1,2)$.  Then, setting $\beta = 2-\alpha - \alpha^2$
and $\gamma = \alpha^2-2$, we have $u(x) = (x-\alpha)(x-\beta)(x-\gamma)$.
The Galois group of $K = \mathbb Q(\alpha)$ is cyclic and is generated by the 
element $\sigma \in G$ such that $\sigma(\alpha) = \beta$.  If we let 
$\textup{\textbf{v}} = [x^3+2xy^2 -1,-xy^2, y-xy^2]^T$, then the
factorization obtained by 
Theorem \ref{almosttherethm} is given by
\[ f = \textup{\textbf{v}}^T
\left[\begin{array}{ccc}1 & -1/2 & 3/2 \\1 & 0 & 2 \\1 & 1/2 & 5/2 \\0 & -1 & 2 \\0 & -\sqrt{6}/2 & \sqrt{6}/2 \\0 & -\sqrt{3} & 0\end{array}\right]^T
\left[\begin{array}{ccc}1 & -1/2 & 3/2 \\1 & 0 & 2 \\1 & 1/2 & 5/2 \\0 & -1 & 2 \\0 & -\sqrt{6}/2 & \sqrt{6}/2 \\0 & -\sqrt{3} & 0\end{array}\right]
  \textup{\textbf{v}}.\] 
One can check that this factorization already produces the rational sum of squares representation
we encountered  in Example \ref{cubeex}.\qed
\end{ex}

We note that when $K$ is an arbitrary number field, Galois over $\mathbb Q$, 
our approach still produces a result similar in spirit to Theorem \ref{almosttherethm}.
The only difference is that we must allow negative integers $l_k$ 
in the statement.

\begin{thm}\label{almosttherethm2}
Let $K$ be a finite Galois extension of $\mathbb Q$ and set $r = [K: \mathbb Q]$.
Then for any $p \in K[x_1,\ldots,x_n]$, there are integers $l_1,\ldots,l_r$
such that 
\[ \sum_{\sigma \in G} (\sigma p)^2 = q^TCC^Tq,\]
in which $q$ is a vector of polynomials in $\mathbb Q [x_1,\ldots,x_n]$ and
$C$ is an $r \times {r+1 \choose 2}$ matrix with entries in 
$F = \mathbb Q(\sqrt{l_1},\ldots,\sqrt{l_r},\sqrt{2})$.
\end{thm}

The following corollary is the closest we come to answering Sturmfels' question in 
the general case.  It follows from applying Theorem \ref{almosttherethm2}
in the same way that Theorem \ref{almosttherethm} will be used below to prove
Theorem \ref{sqrtsfieldext}.

\begin{cor}
Let $K$ be a finite extension of $\mathbb Q$.  
 If $f \in \mathbb Q[x_1,\ldots,x_n]$ is a sum of squares over 
$K[x_1,\ldots,x_n]$, then it is a difference of two sums of squares over $\mathbb Q[x_1,\ldots,x_n]$.
\end{cor}

\begin{ex}
Consider the degree $2$ field extension $K = \mathbb Q(i \sqrt{2})$,
which is the splitting field of $u(x) = x^2+2$. One can check that setting $(b_1,b_2) = (0,1)$, 
$\delta = -1$, and $d = [\sqrt{2},i\sqrt{3}]^T$ in Theorem \ref{symcomp} produces the symmetric matrix
\[ A = \left[\begin{array}{ccc} 2 & i\sqrt{6}  \\ i \sqrt{6} & -2 \end{array}\right].\]
It follows that the $2 \times 2$ Vandermonde matrix $V_2$ as in (\ref{vandtheta}) 
satisfies $V_2 V_2^T = CC^T$, in which 
\[C = \left[\begin{array}{ccc}1 & 1 & 0 \\2 & -2 & 2i\sqrt{3}\end{array}\right].\]
This calculation expresses the polynomial $f = (x+i\sqrt{2} y)^2 + (x-i\sqrt{2} y)^2$ 
as the difference
\[f = (x+2y)^2 + (x-2y)^2 - 12y^2.\]
\qed
\end{ex}

\section{Proof of Theorem \ref{weakmainthm}}\label{proof}

In this section, we complete the proof of our main theorem.  The results
so far show that if $f$ is a sum of $m$ squares in $K[x_1,\ldots,x_n]$ for 
a totally real field $K$, Galois over $\mathbb Q$, then $f$ is a sum of 
$m \cdot {[K:\mathbb Q] +1 \choose 2}$ squares in $L[x_1,\ldots,x_n]$, 
where $L = \mathbb Q(\sqrt{l_1},\ldots,\sqrt{l_r},\sqrt{2})$ 
for some positive integers $l_k$.  The proof of Theorem \ref{weakmainthm}
is thus complete if we can show the following.

\begin{thm}\label{sqrtsfieldext}
Let $l_1,\ldots,l_{r+1}$ be positive integers and set $L = \mathbb Q(\sqrt{l_1},\ldots,\sqrt{l_{r+1}})$.
If $f \in \mathbb Q [x_1,\ldots,x_n]$ is a sum of $s$ squares in  $L[x_1,\ldots,x_n]$,
then $f$ is a rational sum of at most $s \cdot 2^{r+1}$ squares in $\mathbb Q [x_1,\ldots,x_n]$.
\end{thm}

\begin{proof}
Let $l$ be a positive integer and let $L = F(\sqrt{l})$ be a quadratic extension of 
a field $F$ of characteristic $0$. We shall prove:  If 
$f \in  F [x_1,\ldots,x_n]$ is a rational sum of 
$s$ squares in  $L [x_1,\ldots,x_n]$,
then $f$ is a rational sum of at most $2s$ squares in 
$F[x_1,\ldots,x_n]$.  The theorem then follows by repeated application of this fact.

If $L = F$, then there is nothing to prove.  Otherwise, let $\sigma \in \text{Gal}(L/F)$ be
such that $\sigma(\sqrt{l}) = -\sqrt{l}$, and let $f \in  F [x_1,\ldots,x_n]$ be 
a sum of $s$ squares in $L [x_1,\ldots,x_n]$:
\begin{equation*}\label{sumsqrL}
\begin{split}
f = \ &  \sum_{i=1}^s p_i^2  =  \frac{1}{2} \sum_{i=1}^s \left(p_i^2 + (\sigma p_i)^2 \right).
\end{split}
\end{equation*}
It therefore suffices to prove that for fixed
$p \in  L[x_1,\ldots,x_n]$, the element $p^2 + (\sigma p)^2$ is 
a rational sum of $2$ squares.  Finally, writing $p = a + b \sqrt{l}$ for 
$a,b \in  F[x_1,\ldots,x_n]$, we have that  \[(a+b\sqrt{l})^2 + (a-b\sqrt{l})^2 = 2a^2 + 2lb^2.\]
This completes the proof.
\end{proof}


\end{document}